\documentclass[16pt]{article}

\usepackage{epsfig}
\usepackage{graphicx}
\usepackage{graphpap}
\usepackage[mathscr]{eucal}
\usepackage{amsmath}
\usepackage{amsfonts}
\usepackage{amssymb}
\usepackage{amsthm}
\usepackage{amstext}
\usepackage{amscd}
\usepackage{makeidx}
\usepackage{graphics}

\newcommand{\noi}{\noindent}

\begin{document}
\title{Metric Dimension of Amalgamation of Graphs}
\author{
\normalsize Rinovia Simanjuntak, Saladin Uttunggadewa, and Suhadi Wido Saputro\\
\normalsize Combinatorial Mathematics Research Group\\
\normalsize Faculty of Mathematics and Natural Sciences\\
\normalsize Institut Teknologi Bandung, Bandung 40132, Indonesia\\
\scriptsize {\small {\bf e-mail}: {\tt \{rino,s\_uttunggadewa,suhadi\}@math.itb.ac.id}}
}

\date{ }

\newtheorem{lemma}{Lemma}
\newtheorem{theorem}{Theorem}
\newtheorem{proposition}{Proposition}
\newtheorem{corollary}{Corollary}
\newtheorem{definition}{Definition}
\newtheorem{observation}{Observation}

\maketitle

\begin{abstract}
A set of vertices $S$ resolves a graph $G$ if every vertex is uniquely determined by its vector of distances to the vertices in $S$. The metric dimension of $G$ is the minimum cardinality of a resolving set of $G$.

Let $\{G_1, G_2, \ldots, G_n\}$ be a finite collection of graphs and each $G_i$ has a fixed vertex $v_{0_i}$ or a fixed edge $e_{0_i}$ called a terminal vertex or edge, respectively. The \emph{vertex-amalgamation} of $G_1, G_2, \ldots, G_n$, denoted by $Vertex-Amal\{G_i;v_{0_i}\}$, is formed by taking all the $G_i$'s and identifying their terminal vertices. Similarly, the \emph{edge-amalgamation} of $G_1, G_2, \ldots, G_n$, denoted by $Edge-Amal\{G_i;e_{0_i}\}$, is formed by taking all the $G_i$'s and identifying their terminal edges.

Here we study the metric dimensions of vertex-amalgamation and edge-amalgamation for finite collection of arbitrary graphs. We give lower and upper bounds for the dimensions, show that the bounds are tight, and construct infinitely many graphs for each possible value between the bounds.
\end{abstract}

\section{Introduction}

\noi In this paper we consider finite, simple, and connected graphs.
The vertex and edge sets of a graph $G$ are denoted by $V(G)$ and
$E(G)$, respectively.

\noi The \textit{distance} $d(u,v)$ between two vertices $u$ and $v$
in a connected graph $G$ is the length of a shortest $u-v$ path in
$G$. For an ordered set $W = \{w_{1}$, $w_{2}$, $\cdots$, $w_{k}\}$
$\subseteq V(G)$, we refer to the $k$-vector $r(v|W)=(d(v,w_{1}), d(v,w_{2}), \cdots, d(v,w_{k}))$ as the {\em(metric)
representation of $v$ with respect to $W$}. The set $W$ is called a
\textit{resolving set} for $G$ if $r(u|W)= r(v|W)$ implies that $u =
v$ for all $u,v \in G$. In a graph $G$, a resolving set with minimum
cardinality is called a \textit{basis} for $G$. The \textit{metric
dimension}, $dim(G)$, is the number of vertices in a basis for $G$.\\

\noi The metric dimension problem was first introduced in 1975 by Slater \cite{Sla}, and independently by Harary and Melter \cite{Har} in 1976; however the problem for hypercube was studied (and solved asymptotically) much earlier in 1963 by Erd\H{o}s and R\'{e}nyi \cite{ER63}. In general, it is difficult to obtain a basis and metric dimension for arbitrary graph. Garey and Johnson \cite{Garey}, and also Khuller \textit{et al}. \cite{Khu96}, showed that determining the metric dimension of an arbitrary graph is an NP-complete problem. The problem is still NP-complete even if we consider some specific families of graphs, such as bipartite graphs \cite{MARR08}, planar graphs \cite{DPSL12}, or Gabriel unit disk graphs \cite{HW13}. Thus research in this area are then constrained towards: characterizing graphs with particular metric dimensions, determining metric dimensions of particular graphs, and constructing algorithm that "best" approximate metric dimensions.\\

\noi Until today, only graphs of order $n$ with metric dimension 1 (the paths), $n-3$, $n-2$, and $n-1$ (the complete graphs) have been characterized \cite{Char2000,HMPSW10,JO}.
On the other hand, researchers have determined metric dimensions for many particular classes of graphs, such as trees \cite{Char2000, Har, Khu96}, cycles \cite{Char2000}, grids \cite{Melter1984}, complete multipartite graphs \cite{Char2000,BBSSS11}, hypercube \cite{ER63,Li64,Char2000,Be13}, wheels \cite{BCPZ03, CHMPPSW05, Sha2002}, fans \cite{CHMPPSW05}, unicyclic graphs \cite{Poisson2002}, honeycombs \cite{Manuel2008,XF}, circulant graphs \cite{IBBJ12}, Jahangir graphs \cite{Tomescu2007}, Sierp\'{\i}ski graphs \cite{KZ}, and classical binomial random graph \cite{BMP13}. Recently in 2011, Bailey and Cameron \cite{BC11} established relationship between the base size of automorphism group of a graph and its metric dimension. This result then motivated researchers to study metric dimensions of distance regular graphs, such as Grassman \cite{BM11, GWL13}, Johnson, Kneser \cite{BCGGMMP13}, and bilinear form graphs \cite{FW12, GWL13}.\\

\noi In the area of constructing algorithm that "best" approximate metric dimensions, researchers have utilized integer programming \cite{CO01}, genetic algorithm \cite{KKC09}, variable neighborhood search based heuristic \cite{MKKC12}, and greedy constant factor approximation algorithm \cite{HSV12}.\\

\noi There are also some results of metric dimensions of graphs resulting from graph operations; for instance: Cartesian product graphs \cite{Melter1984,Khu96,CHMPPSW07,Wido2008b}, joint product graphs \cite{BCPZ03,CHMPPSW05,Sha2002}, strong product \cite{RKYS}, corona product graphs \cite{YKR10,IBR11}, lexicographic product graphs \cite{SRUABSB13}, hierarchical product graphs \cite{FW13}, and line graphs \cite{KY12, FXW13}.\\

\noi In this paper, we study metric dimension of graphs resulting from another type of graph operations, i.e., vertex-amalgamation and edge-amalgamation of a finite collection of arbitrary graphs. Previous study of such graphs has been done for vertex-amalgamation of two arbitrary graphs \cite{Poisson2002} and vertex-amalgamation and edge-amalgamation of particular families of graphs, which include cycles, complete graphs, and prisms \cite{IBSS10a,IBSS10b,SABISU,SM}. We present these known results in the next section and then provide more general results in the last section: give lower and upper bounds for the dimensions, show that the bounds are tight, and construct infinitely many graphs for each possible value between the bounds.

\section{Previous Results}

\noi Let $\{G_1, G_2, \ldots, G_n\}$ be a finite collection of graphs and each \emph{block} $G_i$ has a fixed vertex $v_{0_i}$ or a fixed edge $e_{0_i}$ called a \emph{terminal vertex} or \emph{edge}, respectively. The \emph{vertex-amalgamation} of $G_1, G_2, \ldots, G_n$, denoted by $Vertex-Amal\{G_i;v_{0_i}\}$, is formed by taking all the $G_i$'s and identifying their terminal vertices. Similarly, the \emph{edge-amalgamation} of $G_1, G_2, \ldots, G_n$, denoted by $Edge-Amal\{G_i;e_{0_i}\}$, is formed by taking all the $G_i$'s and identifying their terminal edges.

\noi In \cite{Poisson2002}, Poisson and Zhang studied vertex-amalgamation of two nontrivial connected graphs $G_1, G_2$ and provide a lower bound as follow.

\begin{theorem} \emph{\cite{Poisson2002}}
Let $G$ be the vertex-amalgamation of nontrivial connected graphs $G_1$ and $G_2$ with terminal vertices $v_{0_1}$ and $v_{0_2}$. Then
\[dim(G) \geq dim(G_1) + dim(G_2) - 2.\]
\label{2}
\end{theorem}

\noi Other known results are vertex-amalgamation and edge-amalgamation of particular families of graphs, as presented in the following theorems. We denote by $C_n$ the cycle of order $n$, by $K_n$ the complete graph of order $n$, and by $Pr_n$ the prism of order $2n$.

\begin{theorem} \emph{\cite{IBSS10b,SABISU}}
Let $\{C_{c_1}, C_{c_2}, \ldots, C_{c_n}\}$ be a collection of $n$ cycles with $n_e$ cycles of even order. Suppose that $G$ is the vertex-amalgamation of $C_{c_1}, C_{c_2}, \ldots, C_{c_n}$ and $H$ is the edge-amalgamation of $C_{c_1}, C_{c_2}, \ldots, C_{c_n}$. Then
\[dim(G)=\left\{
\begin{array}{ll}
\sum_{i=1}^{n} dim(C_{c_i}) - n & , n_e=0,\\
\sum_{i=1}^{n} dim(C_{c_i}) - n + n_e - 1& , n_e \geq 1
\end{array}
\right.\]
and\\
\[\sum_{i=1}^{n} dim(C_{c_i}) - n - 2 \leq dim(H) \leq \sum_{i=1}^{n} dim(C_{c_i}) - n.\]
\label{Cn}
\end{theorem}

\begin{theorem} \emph{\cite{SM}} Let $\{K_{k_1}, K_{k_2}, \ldots, K_{k_n}\}$ be a collection of $n$ complete graphs with $n_2$ complete graphs of order $2$ and $n_3$ complete graphs of order $3$. Suppose that $G$ is the vertex-amalgamation of $K_{k_1}, K_{k_2}, \ldots, K_{k_n}$ and $H$ is the edge-amalgamation of $K_{k_1}, K_{k_2}, \ldots, K_{k_n}$. Then
\[dim(G)=\left\{
\begin{array}{ll}
\sum_{i=1}^{n} dim(K_{k_i}) - n + n_2 - 1 & , n_2 \geq 2,\\
\sum_{i=1}^{n} dim(K_{k_i}) - n & , \hbox{otherwise}
\end{array}
\right.\]
and\\
\[dim(H)=\left\{
\begin{array}{ll}
\sum_{i=1}^{n} dim(K_{k_i}) - 2n + 1 & , n_3=0 \hbox{ or } n=2 \hbox{ and } n_3=1,\\
\sum_{i=1}^{n} dim(K_{k_i}) - 2n & , \hbox{otherwise}.
\end{array}
\right.\]
\label{Kn}
\end{theorem}

\begin{theorem} \emph{\cite{SM}} Let $\{Pr_{p_1}, Pr_{p_2}, \ldots, Pr_{p_n}\}$ be a collection of $n$ prisms with $n_o$ prisms of odd order. Suppose that $G$ is the vertex-amalgamation of $Pr_{p_1}, Pr_{p_2}, \ldots, Pr_{p_n}$ and $H$ is the edge-amalgamation of $Pr_{p_1}, Pr_{p_2}, \ldots, Pr_{p_n}$. Then
\[dim(G)=\left\{
\begin{array}{ll}
\sum_{i=1}^{n} dim(Pr_{p_i}) - n & , n_o=0,\\
\sum_{i=1}^{n} dim(Pr_{p_i}) - n + n_o - 1& , n_o \geq 1
\end{array}
\right.\]
and\\
\[dim(H)=\sum_{i=1}^{n} dim(Pr_{p_i}) - n + n_o - 1.\]
\label{Prn}
\end{theorem}

\section{Main Results}

\noi The next theorem provide the sharp lower and upper bounds for the metric dimension of vertex-amalgamation of finite collection of arbitrary graphs, as well as a construction showing that all values between the bound are attainable.
\begin{theorem}
Let $\{G_1, G_2, \ldots, G_n\}$ be a finite collection of graphs and $v_{0_i}$ is a terminal vertex of $G_i$, $i=1,2,\ldots,n$. If $G$ is the vertex-amalgamation of $G_1, G_2, \ldots, G_n$, $Vertex-Amal\{G_i;v_{0_i}\}$, then
\[\sum_{i=1}^{n} dim(G_i) - n \leq dim(G) \leq \sum_{i=1}^{n} dim(G_i) + n - 1.\]
Moreover, the bounds are sharp and there are infinitely many graphs with dimension equal to all values within the range of the bounds.
\label{VA}
\end{theorem}
\begin{proof}
For the lower bound, consider a vertex set $W$ with cardinality less than $\sum_{i=1}^{n} dim(G_i) - n$. Consequently, there exists a block $G_i$ which the cardinality of its intersection with $W$ is less than $dim(G_i) - 1$. Therefore $W$ could not be a resolving set of $G$ and so \[dim(G) \geq \sum_{i=1}^{n} dim(G_i) - n.\]

\noi For the upper bound, consider two arbitrary blocks $G_i$ and $G_j$ of $G$ and their basis $R_i$ and $R_j$. Clearly, at most two vertices in $G_i \bigcup G_j$, say $x$ and $y$, could have the same representation with respect to $R_i \bigcup R_j$, since otherwise there exist two vertices in a block, say $G_i$, having the same representation with respect to $R_i$, a contradiction with $R_i$ being a resolving set. Thus, to guarantee all vertices in $G_i \bigcup G_j$ have have different representation, we have to add either $x$ or $y$ to $R_i \bigcup R_j$. If we consider each pair of blocks in $G$, we obtain
\[dim(G) \leq \sum_{i=1}^{n} dim(G_i) + n - 1.\]

\noi Now let us start our construction by considering $\{G_1, G_2, \ldots, G_n\}$ as a finite collection of complete graphs of order at least 3. By Theorem \ref{Kn}, $dim(G)=\sum_{i=1}^{n} dim(G_i) - n$, which achieve the lower bound. We then replace $G_1$ with a path consisting non-leaf terminal vertex. Let $B$ be the union of all the blocks' basis. Since the path has dimension 1 and its basis vertex is a leaf vertex, then the two vertices of the path adjacent to the terminal vertex will have the same representation with respect to $B$. Thus we have to add one vertex, i.e. one of the two vertices of the path adjacent to the terminal vertex, to $B$ in order to obtain a basis for $G$. This results in $dim(G)=\sum_{i=1}^{n} dim(G_i) - n + 1$, which increases the lower bound by one. We continue this process by replacing the $G_i$s one at a time until all complete graphs are replaced with paths (see Figure \ref{ULBVA}). The resulting graph is a subdivided star, whose dimension achieves the upper bound.
\end{proof}

\begin{figure}[h]
\centerline{\includegraphics[width=0.9\textwidth, height=70px]{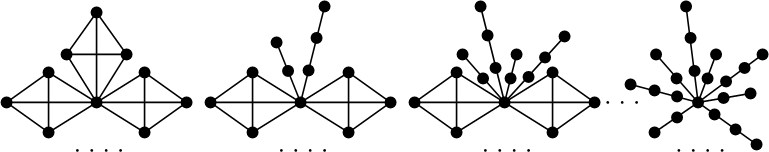}}
\caption{\small Vertex-amalgamations of graphs whose dimensions attaining all values between the lower and upper bounds.}
\protect\label{ULBVA}
\end{figure}

\noi Note that the lower bound in the previous theorem generalizes the result of Poisson and Zhang in Theorem \ref{2}. From Theorems \ref{Cn} and \ref{Prn}, we can see that there exist amalgamations of particular cycles and prisms whose dimensions attaining the lower bounds. These graphs could be used in the construction of the proof of Theorem \ref{VA}.

\noi To prove the result for edge-amalgamation of a finite collection of graphs, we need to know the dimensions of two special graphs. The first graph is complete bipartite graphs $K_{m,n}$. It is known that $dim(K_{m,n})=m+n-2$ and the basis consists of all vertices in $K_{m,n}$ except for one vertex from each partite set. The second graph is a variation of a cycle of order $n$, $C_n$. Suppose that $V(C_n)=\{x_1, x_2, \ldots, x_n\}$ and $E(C_n)=\{x_n x_1, x_i x_{i+1}, i=1, 2, \ldots, n-1\}$. We add two vertices $y_2$, $y_5$ and six edges $y_2 x_i, i=1, 2, 3$, $y_5 x_i, i=4, 5, 6$. We call the resulting graph a \emph{double-hats cycle}, denoted by $DHC_n$ (see Figure \ref{DHCn}). It is easy to see that a resolving set of $DHC_n$ must consist two vertices: either $x_2$ or $x_5$ and either $y_2$ or $y_5$. On the other hand, the set $\{x_2,y_5\}$ is a resolving set of $DHC_n$, and so $dim(DHC_n)=2$.

\begin{figure}[h]
\centerline{\includegraphics[width=0.2\textwidth, height=60px]{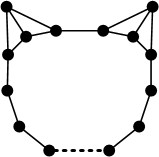}}
\caption{\small The double-hats cycle, $DHC_n$.}
\protect\label{DHCn}
\end{figure}

\begin{theorem}
Let $\{G_1, G_2, \ldots, G_n\}$ be a finite collection of graphs and $e_{0_i}$ is a terminal edge of $G_i$, $i=1,2,\ldots,n$. If $H$ is the edge-amalgamation of $G_1, G_2, \ldots, G_n$, $Edge-Amal\{G_i;e_{0_i}\}$, then
\[\sum_{i=1}^{n} dim(G_i) - 2n \leq dim(H) \leq \sum_{i=1}^{n} dim(G_i) + n - 1.\]
Moreover, the bounds are sharp and there are infinitely many graphs with dimension equal to all values within the range of the bounds.
\label{EA}
\end{theorem}
\begin{proof}
For the lower bound, consider a vertex set $W$ with cardinality less than $\sum_{i=1}^{n} dim(G_i) - 2n$. Consequently, there exists a block $G_i$ which the cardinality of its intersection with $W$ is less than $dim(G_i) - 2$. Therefore $W$ could not be a resolving set of $H$ and so \[dim(G) \geq \sum_{i=1}^{n} dim(G_i) - 2n.\]

\noi For the upper bound, consider two arbitrary blocks $G_i$ and $G_j$ of $G$ and their basis $R_i$ and $R_j$. Clearly, at most two vertices in $G_i \bigcup G_j$, say $x$ and $y$, could have the same representation with respect to $R_i \bigcup R_j$, since otherwise there exist two vertices in a block, say $G_i$, having the same representation with respect to $R_i$, a contradiction with $R_i$ being a resolving set. Thus, to guarantee all vertices in $G_i \bigcup G_j$ have have different representation, we have to add either $x$ or $y$ to $R_i \bigcup R_j$. If we consider each pair of blocks in $G$, we obtain
\[dim(G) \leq \sum_{i=1}^{n} dim(G_i) + n - 1.\]

\noi Similarly to the construction for vertex-amalgamation of graphs in the proof of Theorem \ref{VA}, we start by considering $\{G_1, G_2, \ldots, G_n\}$ as a finite collection of symmetric complete bipartite graphs $K_{m_i,m_i}$ with vertex-set partitioned into $\{x_1, x_2, \ldots, x_{m_i}\}$ and $\{y_1, y_2, \ldots, x_{m_i}\}$. Let $x_{m_i}y_{m_i}$ be the terminal edge in each $K_{m_i,m_i}$. By the first part of the theorem, we have $dim(H) \geq \sum_{i=1}^{n} dim(G_i) - 2n$. Now, consider the set $R=\bigcup_{i=1}^n \{x_1, x_2, \ldots, x_{m_i-2}\}$. It is easy to see that $R$ is a resolving set for $H$, and thus $dim(H) = \sum_{i=1}^{n} dim(G_i) - 2n$. Therefore we have amalgamation of graphs attaining the lower bound.

\noi We then replace $G_1$ with a double-hats cycle $DHC_n$ with terminal edge $x_6x_7$ (refer to the standard vertices notation of $DHC_n$). Let $B$ be the union of all the blocks' basis. The two vertices of $DHC_n$ adjacent to the terminal edge will have the same representation with respect to $B$. Thus we have to add one vertex, i.e. one of the two vertices of $DHC_n$ adjacent to the terminal edge, to $B$ in order to obtain a basis for $H$. This results in $dim(H)=\sum_{i=1}^{n} dim(G_i) - n + 1$, which increases the lower bound by one. We continue this process by replacing the $G_i$s one at a time until all complete bipartite graphs are replaced with $DHC_n$s. The dimension of the resulting graph then achieves the upper bound.
\end{proof}

\noi Notice that there exists edge-amalgamation of some complete graphs with dimension equal to the lower bound (see Theorem \ref{Kn}), and so these graphs could be used in the construction of the proof of the previous theorem. We could also see from Theorems \ref{Cn}, \ref{Kn}, and \ref{Prn}, that the dimensions of edge-amalgamations of cycles and prims are the middle values between the lower and upper bounds, while the dimensions of edge-amalgamation of complete graphs are values around the lower bound. Determining which collection of graphs whose vertex-amalgamation and edge-amalgamation have small dimensions (close to the lower bounds) might be seen as interesting problems.

\end{document}